\documentclass{amsart}%
\usepackage{palatino, mathpazo}
\usepackage{amsfonts}
\usepackage{amsmath}
\usepackage{amssymb,latexsym,xcolor}
\usepackage{graphicx}
\usepackage{harpoon}
\usepackage[mathscr]{eucal}
\usepackage{amssymb}%
\usepackage[
linktocpage=true,colorlinks,citecolor=magenta,linkcolor=blue,urlcolor=magenta]{hyperref}
\providecommand{\U}[1]{\protect \rule{.1in}{.1in}}

\newtheorem{theorem}{Theorem}[section]

\newtheorem{corollary}[theorem]{Corollary}
\newtheorem{definition}[theorem]{Definition}

\newtheorem{Theorem}{Theorem}

\theoremstyle{remark}
\newtheorem{remark}[theorem]{Remark}

\numberwithin{equation}{section}

\newcommand{\Z}{{\mathbb Z}}

\newcommand{\N}{{\mathbb N}}

\begin{document}
\title[Curvature equation and Painlev\'{e} VI equation]{Sovability of curvature equations with multiple singular sources on torus via Painlev\'{e} VI equations}
\author{Zhijie Chen}
\address{Department of Mathematical Sciences, Yau Mathematical Sciences Center, Tsinghua University, Beijing, China}
\email{zjchen2016@tsinghua.edu.cn}
\author{Ting-Jung Kuo}
\address{Department of Mathematics, National Taiwan Normal University, Taipei 11677,
Taiwan }
\email{tjkuo1215@ntnu.edu.tw}
\author{Chang-Shou Lin}
\address{Department of Mathematics, National Taiwan University, Taipei 10617, Taiwan }
\email{cslin@math.ntu.edu.tw}

\begin{abstract}
We study the curvature equation with multiple singular sources on a torus
\[\Delta u+e^{u}=8\pi \sum_{k=0}^{3}n_{k}\delta_{\frac{\omega_{k}}{2}}%
+4\pi \left(  \delta_{p}+\delta_{-p}\right) \quad \text{ on }\;E_{\tau}:=\mathbb{C}/(\mathbb Z+\mathbb{Z}\tau),\]
where $n_k\in\mathbb N$ and $\delta_a$ denotes the Dirac measure at $a$. This is known as a critical case for which the apriori estimate does not hold, and the existence of solutions has been a long-standing problem. In this paper, by establishing a deep connection with Painlev\'{e} VI equations, we show that the existence of even solutions (i.e. $u(z)=u(-z)$) depends on the location of the singular point $p$, and we give a sharp criterion of $p$ in terms of Painlev\'{e} VI equations.
\end{abstract}
\maketitle

\section{Introduction}

Let $\tau \in \mathbb{H}=\left \{  \tau\in\mathbb C|\operatorname{Im}\tau>0\right \}$, $\Lambda_{\tau}=\mathbb{Z}+\mathbb{Z}\tau$, and denote
$$\omega_{0}=0,\quad\omega_{1}=1,\quad\omega_{2}=\tau,\quad\omega_{3}=1+\tau.$$Let $E_{\tau}:=\mathbb{C}/\Lambda_{\tau}$ be a flat torus in the
plane and $E_{\tau}[2]:=\{ \frac{\omega_{k}}{2}|k=0,1,2,3\}+\Lambda
_{\tau}$ be the set consisting of the lattice points and half periods
in $E_{\tau}$. For $z\in\mathbb{C}$ we denote $[z]:=z \ (\text{mod}\ \Lambda_{\tau}) \in E_{\tau}$.
For a point $[z]$ in $E_{\tau}$ we often write $z$ instead of $[z]$ to
simplify notations when no confusion arises.

Let $\wp(z)=\wp(z;\tau)$ be the
Weierstrass elliptic function with periods $\Lambda_{\tau}$, defined by \[
\wp(z)=\wp(z;\tau)=\frac{1}{z^{2}}+\sum_{\omega \in \Lambda_{\tau}\backslash \left \{
0\right \}  }\left(  \frac{1}{(z-\omega)^{2}}-\frac{1}{\omega^{2}}\right).
\] Let $e_k(\tau):=\wp(\frac{\omega_k}{2};\tau)$, $k=1,2,3$. Then it is well known that
\[\wp'(z;\tau)^2=4\prod_{k=1}^3(\wp(z;\tau)-e_k(\tau))=4\wp(z;\tau)^3-g_2(\tau)\wp(z;\tau)-g_3(\tau),\]
where $g_2(\tau), g_3(\tau)$ are invariants of the elliptic curve $E_{\tau}$.
Let $\zeta(z)=\zeta(z;\tau):=-\int^{z}\wp(\xi;\tau)d\xi$
be the Weierstrass zeta function with two quasi-periods $\eta_k=\eta_{k}(\tau)$:
\begin{equation}
\eta_k=\eta_{k}(\tau):=2\zeta(\tfrac{\omega_{k}}{2} ;\tau)=\zeta(z+\omega_{k} ;\tau)-\zeta(z;\tau),\quad k=1,2.
\label{40-2}
\end{equation}
 Notice that $\zeta(z)$ is an odd
meromorphic function with simple poles at $\Lambda_{\tau}$.

In this paper, we study the curvature equation with multiple singular sources
\begin{equation}\label{mfe-p}
\Delta u+e^{u}=8\pi \sum_{k=0}^{3}n_{k}\delta_{\frac{\omega_{k}}{2}}%
+4\pi \left(  \delta_{p}+\delta_{-p}\right) \quad \text{ on }\;E_{\tau},
\end{equation}
where $\delta_x$ denotes the Dirac measure at $x\in E_{\tau}$, $p\in E_{\tau}\setminus E_{\tau}[2]$ (i.e. $p\neq \frac{\omega_k}{2}$ in $E_{\tau}$),  and $n_{k}\in\N$ for all $k$.
Geometrically, \eqref{mfe-p} arises from the prescribing Gaussian curvature problem. A solution $u$ of \eqref{mfe-p} leads to a metric $g=\frac12 e^uds^2$ with constant Gaussian curvature $+1$ acquiring conic singularities at $\frac{\omega_k}{2}$ and $\pm p$. Physically, \eqref{mfe-p}
appears in statistical physics as the equation for the mean field
limit of the Euler flow in Onsager's vortex model (cf. \cite{CLMP}), hence
also called a mean field equation. It is also related to the
self-dual condensates of the Chern-Simons-Higgs model in superconductivity.
We refer the readers to \cite{CLW,CL-1,CL-2,CL-3,CKL2,CL-AJM,EG-2015,EG,LW,LW2,LY,MR,NT1} and references therein for
recent developments on related subjects of (\ref{mfe-p}). Remark that if $p\in E_{\tau}[2]$, then \eqref{mfe-p} reduces to
\begin{equation}\label{mfe-0}
\Delta u+e^{u}=8\pi \sum_{k=0}^{3}n_{k}\delta_{\frac{\omega_{k}}{2}} \quad \text{ on }\;E_{\tau},
\end{equation}
with $\max_{k}n_k\geq 1$, which has been studied in \cite{CLW,CKL2,CL-AJM,EG-2015,EG,LW,LW2}. In particular, it turns out that whether \eqref{mfe-0} has even solutions (i.e. $u(z)=u(-z)$) or not depends on the choice of $\tau$, and those $\tau$'s such that \eqref{mfe-0} has even solutions can be characterized as zeros of certain pre-modular forms; see \cite{CKL2,LW2} for details.

In this paper, we focus on the case $p\notin E_{\tau}[2]$.
From the PDE point of view, since the apriori estimate does not hold for \eqref{mfe-p} (see Remark \ref{rmk3-4} below), neither the Leray-Schauder degree method \cite{CL-2,CL-3} nor the variational method \cite{MR} works, and 
whether solutions of \eqref{mfe-p} exist or not remains completely open. The purpose of this paper is to show that \emph{the existence of even solutions of \eqref{mfe-p} depends on the location of the singular point $p$, and to give a sharp criterion of $p$ in terms of Painlev\'{e} VI equations}.

The well-known Painlev\'{e} VI equation (PVI) is a second order nonlinear ODE with complex variables and four free parameters
$(\alpha,\beta,\gamma,\delta)$:
\begin{align}
\frac{d^{2}\lambda}{dt^{2}}= &  \frac{1}{2}\left(  \frac{1}{\lambda}+\frac
{1}{\lambda-1}+\frac{1}{\lambda-t}\right)  \left(  \frac{d\lambda}{dt}\right)
^{2}-\left(  \frac{1}{t}+\frac{1}{t-1}+\frac{1}{\lambda-t}\right)
\frac{d\lambda}{dt}\nonumber \\
&  +\frac{\lambda(\lambda-1)(\lambda-t)}{t^{2}(t-1)^{2}}\left[  \alpha
+\beta \frac{t}{\lambda^{2}}+\gamma \frac{t-1}{(\lambda-1)^{2}}+\delta
\frac{t(t-1)}{(\lambda-t)^{2}}\right],\quad t\in\mathbb C.\label{46}%
\end{align}
Historically, PVI was originated from the research on complex ODEs from the
middle of the 19th century up to early 20th century, led by many famous
mathematicians including Painlev\'{e} and his school. The aim is to classify
those nonlinear ODEs whose solutions have the so-called \textit{Painlev\'{e}
property}, namely any solution
has neither movable branch points nor movable essential singularities.

Due to increasingly important roles in both mathematics and physics, PVI has
been widely studied since the early 1970's. See e.g. 
\cite{Bob-E,Dubrovin-Mazzocco,Guzzetti,Hit1,GP,Lisovyy-Tykhyy,Okamoto1,Takemura}
and references therein. Compared with all these
references dealing with PVI, the purpose of this paper is to explore the deep
connection between PVI and the curvature equation \eqref{mfe-p}.

From the Painlev\'{e} property, any solution $\lambda(t)$ of (\ref{46}) is a
multi-valued meromorphic function in $\mathbb{C}\backslash\{
0,1\}$. Therefore, it is natural to lift (\ref{46}) to the covering
space $\mathbb{H}$ of $\mathbb{C}\backslash\{0,1\}  $ by the
following transformation:%
\begin{equation}
t(\tau)  =\frac{e_{3}(\tau)-e_{1}(\tau)}{e_{2}(\tau)-e_{1}(\tau
)}\quad\text{ and }\quad\lambda(t)=\frac{\wp(p(\tau);\tau)-e_{1}(\tau)}{e_{2}%
(\tau)-e_{1}(\tau)}.\label{tr}%
\end{equation} 
Then $\lambda(t)$ solves PVI if and only if $p(\tau)$ satisfies the following
\emph{elliptic form} of PVI (EPVI):
\begin{equation}
\frac{d^{2}p(\tau)}{d\tau^{2}}=\frac{-1}{4\pi^{2}}\sum_{k=0}^{3}\alpha_{k}%
\wp^{\prime}\left( p(\tau)+\tfrac{\omega_{k}}{2};
\tau \right)  , \label{124}%
\end{equation}
where $\wp^{\prime}(z;\tau)=\frac{d}{dz}\wp(z;\tau)$ and
$
(\alpha_{0},\alpha_{1},\alpha_{2},\alpha_{3})  =(
\alpha,-\beta,\gamma,\tfrac{1}{2}-\delta)$.
See e.g. \cite{Babich-Bordag} for the proof. The advantage of (\ref{124}) is that $\wp(p(\tau);\tau)$ is single-valued for
$\tau \in \mathbb{H}$, although $p(\tau)$ has branch points at those $\tau_{0}$
such that $p(\tau_{0})\in E_{\tau_{0}}[2]$.

\begin{remark}
\label{identify}Clearly for any $m_{1},m_{2}\in \mathbb{Z}$, $\pm p(\tau
)+m_{1}+m_{2}\tau$ is also a solution of the elliptic form (\ref{124}). Since
they all give the same solution $\lambda(t)$ of PVI via \eqref{tr}, we
always identify all these $\pm p(\tau)+m_{1}+m_{2}\tau$ with the
same one $p(\tau)$.
\end{remark}

From now on we consider PVI with parameters
\begin{align}
(\alpha,\beta,\gamma,\delta)=  &  \left(  \tfrac{1}{2}(n_{0}+\tfrac{1}{2}%
)^{2},\text{ }-\tfrac{1}{2}(n_{1}+\tfrac{1}{2})^{2},\text{ }\tfrac{1}{2}%
(n_{2}+\tfrac{1}{2})^{2},\right. \nonumber \\
&  \left.  \tfrac{1}{2}-\tfrac{1}{2}(n_{3}+\tfrac{1}{2})^{2}\right),\quad n_{k}\in \mathbb{N}\text{ for all }k, \label{parameter}%
\end{align}
and denoted it by PVI$_{\mathbf{n}}$, where $\mathbf{n}=(n_0,n_1,n_2,n_3)$,
or equivalently
EPVI with parameters
\begin{equation}
\alpha_{k}=\tfrac{1}{2}(n_{k}+\tfrac{1}{2})^{2},\quad n_{k}\in
\mathbb{N}\text{ for all }k,
\label{parameter0}%
\end{equation}
and denoted it by EPVI$_{\mathbf{n}}$.

First we recall Hitchin's famous formula for the case $\mathbf{n}=\mathbf{0}:=(0,0,0,0)$. For any $(r,s)  \in \mathbb{C}^{2}\backslash \frac
{1}{2}\mathbb{Z}^{2}$, where $\frac12\Z^2=\{(r,s) : 2r, 2s\in\Z\}$, we define
\begin{align}\label{z-rs}Z_{r,s}(\tau):=&\zeta(r+s\tau;\tau)-r\eta_1(\tau)-s\eta_2(\tau),\end{align}
and $p_{r,s}^{\mathbf{0}}(\tau
)$ by
\begin{equation}
\wp(p_{r,s}^{\mathbf{0}}(\tau);\tau):=\wp(r+s\tau;\tau)+\frac{\wp^{\prime
}(r+s\tau;\tau)}{2Z_{r,s}(\tau)  }. \label{513-1}%
\end{equation}

\begin{remark}
It is well known that $\wp(\cdot;\tau): E_{\tau}\to \mathbb{CP}^1$ is a double cover, i.e. for any $c\in \mathbb{CP}^1$, there is a unique pair $\pm z_c\in E_{\tau}$ such that $\wp(\pm z_{c};\tau)=c$. Consequently, for any $(r,s)\in \mathbb{C}^{2}\backslash \frac{1}{2}\mathbb{Z}^{2}$, there exists a unique pair $\pm p_{r,s}^{\mathbf{0}}(\tau)\in E_{\tau}$ such that \eqref{513-1} holds.
Therefore, $p_{r,s}^{\mathbf{0}}(\tau)$ is well-defined up to a sign.
\end{remark}

In \cite{Hit1} Hitchin proved the following remarkable result for EPVI$_{\mathbf{0}}$.

\begin{Theorem} \cite{Hit1} \label{thm-Hitchin} For any $(
r,s)  \in \mathbb{C}^{2}\backslash \frac{1}{2}\mathbb{Z}^{2}$\textit{,
}$p_{r,s}^{\mathbf{0}}(\tau)$ given by \eqref{513-1} is a solution to
EPVI$_{\mathbf{0}}$;
or equivalently, $\lambda_{r,s}^{\mathbf{0}}(t):=
\frac{\wp(p_{r,s}^{\mathbf{0}}(\tau);\tau)-e_{1}(  \tau)  }%
{e_{2}(\tau)-e_{1}(\tau)}$ via \eqref{513-1} is a solution to PVI$_{\mathbf{0}}$.\end{Theorem}

It is known (cf. \cite[Section 5]{Chen-Kuo-Lin}) that solutions of EPVI$_{\mathbf{n}}$ could be
obtained from solutions of EPVI$_{\mathbf{0}}$ via the well-known Okamoto transformations
(\cite{Okamoto1}).
\medskip

\noindent{\bf Notation:} \emph{Denote by $p_{r,s}^{\mathbf{n}}(\tau)$ to be the solution of EPVI$_{\mathbf{n}}$ obtained from $p_{r,s}^{\mathbf{0}}(\tau)$
in Theorem \ref{thm-Hitchin} via the Okamoto transformations}.
\medskip

Then by applying Hitchin's formula (\ref{513-1}) and the Okamoto transformation, we proved in \cite[Remark 5.2]{Chen-Kuo-Lin} that

\medskip

\noindent{\bf Lemma B.} \cite[Remark 5.2]{Chen-Kuo-Lin} {\it Given $\mathbf{n}$,
there is a rational function $\Xi_{\mathbf{n}}(\cdot,\cdot,\cdot,\cdot,\cdot,\cdot)$ of six independent variables with coefficients in $\mathbb{Q}$ such that for any $(r,s)\in\mathbb{C}^2\setminus\frac12\mathbb{Z}^2$, there holds
\begin{equation}\label{fxfx}
\wp(p_{r,s}^{\mathbf{n}}(\tau);\tau)= \Xi_{\mathbf{n}}(Z_{r,s}(\tau),\wp(r+s\tau;\tau),\wp'(r+s\tau;\tau),
e_1(\tau),e_2(\tau),e_3(\tau)).
\end{equation}

For example, by writing \[Z=Z_{r,s}(\tau), \quad\wp=\wp(r+s\tau;\tau), \quad\wp'=\wp'(r+s\tau;\tau)\] for convenience, we have
\begin{equation*}
\wp(p_{r,s}^{(1,0,0,0)}(\tau);\tau)=\wp+\frac{3\wp^{\prime}Z^{2}+\left(
12\wp^{2}-g_{2}\right)  Z+3\wp\wp^{\prime}}{2(Z^{3}-3\wp
Z-\wp^{\prime})}.
\end{equation*}}

\begin{remark}It is well-known (cf. \cite{GP}) that PVI \eqref{46} is closely related to the isomonodromy theory of a
second order Fuchsian ODE with five regular singularities $\{0,1,\infty,t,\lambda(t)\}$ on $\mathbb{CP}^{1}$. In \cite{Chen-Kuo-Lin0}, we proved that similarly to this well-known case on $\mathbb{CP}^{1}$,
EPVI (\ref{124}) also governs the isomonodromic deformation of
a generalized Lam\'{e} equation (GLE); see Section 2 for a brief overview, where we will see that the $(r,s)$ will appear as the monodromy data of the associated GLE.
\end{remark}

Define
\begin{equation}\label{fc-om}\Omega^{\mathbf n}_{\tau}:=\Big\{\pm p_{r,s}^{\mathbf n}(\tau)\in E_{\tau}\; :\; (r,s)\in \mathbb{R}^{2}\backslash \frac{1}{2}\mathbb{Z}^{2}\Big\}\neq \emptyset.\end{equation}
For example,
\[
\Omega _{\tau }^{\mathbf 0}=\left \{ p\in E_{\tau }\left \vert
\begin{array}{l}
\wp (p;\tau)=\wp (r+s\tau;\tau)+\frac{\wp ^{\prime }(r+s\tau;\tau )}{%
2Z_{r,s}(\tau)}\text{ } \\
\text{for some }(r,s)\in \mathbb{R}^{2}\backslash \frac{1}{2}\mathbb{Z}^{2}%
\end{array}%
\right. \right \}.
\]
Now we can state our main result of this paper.

\begin{theorem}
\label{theorem-1}Suppose $n_{k}\in\N$ for all $k$ and given
$\tau \in \mathbb{H}$, and $p\not\in E_{\tau}[2]$. 
Then
the curvature equation \eqref{mfe-p}
$$
\Delta u+e^{u}=8\pi \sum_{k=0}^{3}n_{k}\delta_{\frac{\omega_{k}}{2}}+4\pi \left(  \delta_{p}+\delta_{-p}\right)  \quad\text{ on }\;
E_{\tau}$$
has even solutions if and only if $p\in\Omega_{\tau}^{\mathbf n}\setminus E_{\tau}[2]$.

Furthermore, for the special case $\mathbf n=\mathbf 0$, the condition ``even'' can be deleted, namely
\eqref{mfe-p} with $\mathbf n=\mathbf 0$ has solutions if and only if $p\in\Omega_{\tau}^{\mathbf 0}\setminus E_{\tau}[2]$.
\end{theorem}

Theorem \ref{theorem-1} shows that the existence of even solutions of \eqref{mfe-p} essentially depends on the location of the singularity $p$, since our next result indicates that we can not expect $\Omega_{\tau}^{\mathbf n}=E_{\tau}$ in general. Denote $B(a, \varepsilon):=\{z\in E_{\tau}: |z-a|<\varepsilon\}$.

\begin{theorem}\label{main-thm2}
Let $\mathbf n=\mathbf 0$ and $\tau\in i\mathbb{R}_{>0}$, i.e. $E_{\tau}$ is a rectangular torus. Then there is $\varepsilon>0$ such that $\Omega_{\tau}^{\mathbf 0}\cap \cup_{k=0}^3B(\frac{\omega_k}{2}, \varepsilon)=\emptyset$. In other words, if $0<|p-\frac{\omega_k}{2}|<\varepsilon$ for some $k$, then the curvature equation \eqref{mfe-p} with $\mathbf n=\mathbf 0$
\begin{equation}\label{mfe-0}\Delta u+e^u=4\pi (\delta_p+\delta_{-p}) \quad\text{on }E_{\tau}\end{equation}
 has no solutions.
\end{theorem}

\begin{remark}
Lin and Wang \cite{LW} proved that
\begin{equation}\label{mfe-8pi}\Delta u+e^u=8\pi \delta_0 \quad\text{on }E_{\tau}\end{equation}
has no solutions for $\tau\in i\mathbb{R}_{>0}$. We want to emphasize that we can not conclude from this fact that \eqref{mfe-0} 
has no solutions for $|p|>0$ small, because solutions of \eqref{mfe-0} might blow up as $p\to 0$. Thus, Theorem \ref{main-thm2} is not a simple corollary of the non-existence result from \cite{LW}. 
\end{remark}

We will prove Theorem \ref{main-thm2} by using the Green function $G(z)=G(z;\tau)$ on the torus $E_{\tau}$, which is defined by
\[
-\Delta G(z;\tau)=\delta_{0}-\frac{1}{\left \vert E_{\tau}\right \vert }\text{
\ on }E_{\tau},\quad
\int_{E_{\tau}}G(z;\tau)=0,
\]
where $\vert E_{\tau}\vert$ is
the area of the torus $E_{\tau}$. It is an even function with the only
singularity at $0$. For $k\in \{1,2,3\}$, since
\begin{align}\label{gde}\nabla G(\frac{\omega_k}{2})=\nabla G(\frac{\omega_k}{2}-\omega_k)=\nabla G(-\frac{\omega_k}{2})=-\nabla G(\frac{\omega_k}{2}),\end{align}
we see that $\frac{\omega_k}{2}$ are always critical points of $G(z)$. Now for any $p\in E_{\tau}\setminus E_{\tau}[2]$, we define
$$G_p(z):=\frac12(G(z-p)+G(z+p)).$$
Then $G_p(z)$ is also even, so the same argument as \eqref{gde} implies that $\frac{\omega_k}{2}$, $k=0,1,2,3$, are all critical points of $G_p(z)$. 

\begin{definition}
A critical point $a\in E_{\tau}$ of $G$ (resp. $G_p$) is called trivial if $a=-a$ in $E_{\tau}$, i.e. $a\in E_{\tau}[2]$.  A critical point $a\in E_{\tau}$ is called nontrivial if $a\neq-a$ in $E_{\tau}$, i.e. $a\notin E_{\tau}[2]$.
\end{definition}

Therefore, nontrivial critical points, if exist, must appear in pairs.
 Lin and Wang \cite{LW} studied the number of critical points of $G(z;\tau)$ for any $\tau$, and proved that
 
 \medskip
\noindent{\bf Theorem C.} \cite{LW}  {\it $G(z;\tau)$ has almost one pair of nontrivial critical points, or equivalently, $G(z;\tau)$ has either $3$ or $5$ critical points (depends on the choice of $\tau$).}
 \medskip
 
 They proved this remarkable result by showing that there is a one-to-one correspondence between pairs of nontrivial critical points of $G(z;\tau)$ and even solutions of the curvature equation \eqref{mfe-8pi}.
Then by proving that \eqref{mfe-8pi} has at most one even solution by PDE methods, they finally obtained that $G(z;\tau)$ has almost one pair of nontrivial critical points. See \cite{LW} for details.

Here we generalize a part of the above result to the following.

\begin{theorem}\label{thm=4-3} For $p\in E_{\tau}\setminus E_{\tau}[2]$, $G_p(z)$ has nontrivial critical points if and only if $p\in \Omega_{\tau}^{\mathbf 0}$, if and only if the curvature equation
\eqref{mfe-0}
has solutions.
\end{theorem}

Motivated by Theorem C, a natural problem is how many pairs of nontrivial critical points $G_p(z)$ might have for $p\in \Omega_{\tau}^{\mathbf 0}$. By following the approach in \cite{LW}, it is not difficult to prove that there is also a one-to-one correspondence between pairs of nontrivial critical points of $G_p(z)$ and even solutions of the curvature equation \eqref{mfe-0}. However, the problem that how many even solutions \eqref{mfe-0} might have 
seems very difficult and we should study it in future.

The rest of this paper is organized as follows. In Section 2, we introduce a generalized Lam\'{e} equation, briefly review its monodromy theory and its deep connection with the Painlev\'{e} VI equation. In Section 3, by exploring the deep connection between the curvature equation and the generalized Lam\'{e} equation, we prove Theorem \ref{theorem-1}. Finally, Theorems \ref{main-thm2} and \ref{thm=4-3} will be proved in Section 4.

\section{Generalized Lam\'{e} equation}

In this section, we briefly review the monodromy theory of 
the generalized Lam\'{e} equation
(denoted by GLE$(\mathbf{n},p,A,\tau)$):
\begin{align} \label{89-1}
y''=\Big[
&\sum_{k=0}^{3}n_{k}(n_{k}+1)\wp(z-\tfrac{\omega_{k}}{2};\tau)+\frac{3}{4}
(\wp(z+p;\tau)+\wp(z-p;\tau))\\
&+A(\zeta(z+p;\tau)-\zeta(z-p;\tau))+B
\Big]  y=:I_{\mathbf{n}}(z;p,A,\tau)y,\nonumber
\end{align}
with $n_{k}\in \mathbb{N}$ for all $k$, $p \not \in E_{\tau}[2]$, $A\in\mathbb{C}$ and%
\begin{equation}
B=A^{2}-\zeta(2p;\tau)A-\frac{3}{4}\wp(2p;\tau)-\sum_{k=0}^{3}n_{k}(n_{k}+1)\wp(
p-\tfrac{\omega_{k}}{2};\tau)  . \label{101}%
\end{equation}
Note that \eqref{101} is equivalent to saying that $\pm p\not \in E_{\tau}[2]$ are always
\emph{apparent singularities} (i.e. any solution of \eqref{89-1} has no logarithmic singularities at $\pm p$; see Remark \ref{rmk2-0} below). See \cite{Chen-Kuo-Lin0} for a proof and also \cite{CKL1,Chen-Kuo-Lin,Takemura} for recent studies on (\ref{89-1}).
Clearly
\begin{align*}
&  \text{GLE}(\mathbf{n},p,A,\tau)\; \text{\textit{is independent of any
representative }}\tilde{p}\in p+\Lambda_{\tau}\\
&  \text{\textit{and} GLE}(\mathbf{n},p,A,\tau)=\text{GLE}(\mathbf{n}%
,-p,-A,\tau).
\end{align*}
The precise connection between GLE$(\mathbf{n},p,A,\tau)$ and the curvature equation \eqref{mfe-p} will be studied in Section 3.

\begin{remark}\label{rmk2-0} Fix any $a\in E_{\tau}$. Suppose GLE$(\mathbf{n},p,A,\tau)$ has a local solution of the following form near $z=a$
\begin{equation}\label{ya}y_{\kappa}(z)=(z-a)^{\kappa}\sum_{j=0}^\infty c_j (z-a)^j=c_0 (z-a)^{\kappa}(1+O(z-a)),\;\,\text{with }c_0\neq 0.\end{equation}
then by inserting this formula into \eqref{89-1}, a direct computation shows that $\kappa$ satisfies $\kappa^2-\kappa+c_{a}=0$, where 
$$c_a=\begin{cases}0 &\text{if }a\notin E_{\tau}[2]\cup\{\pm [p]\},\\
-\frac{3}{4}&\text{if }a=\pm p,\\
-n_k(n_k+1)&\text{if }a=\frac{\omega_k}{2}.\end{cases}$$
The zeros of $\kappa^2-\kappa+c_{a}=0$ are called the \emph{local exponents} of GLE$(\mathbf{n},p,A,\tau)$ at $a$. Therefore, the local exponents of GLE$(\mathbf{n},p,A,\tau)$ at $\frac{\omega_k}{2}$ are $-n_k, n_k+1$, the local exponents of GLE$(\mathbf{n},p,A,\tau)$ at $\pm p$ are $-\frac12, \frac32$, and the local exponents of GLE$(\mathbf{n},p,A,\tau)$ at any $a\notin E_{\tau}[2]\cup\{\pm [p]\}$ are always $0,1$.

From the theory of linear ODEs with complex variable (see e.g. \cite[Chapter 1]{GP}), we collect some basic facts.
\begin{itemize}
\item For $a=\pm p$, GLE$(\mathbf{n},p,A,\tau)$ always has a local solution $y_{\kappa}(z)$ with $\kappa=\frac32$, but GLE$(\mathbf{n},p,A,\tau)$ may not have local solutions $y_{\kappa}(z)$ of the form \eqref{ya} with $\kappa=-\frac12$ in general. If this happens, then GLE$(\mathbf{n},p,A,\tau)$ must have solutions with logarithmic singularity at $\pm p$, i.e. the local expansion of such solutions at $\pm p$ is not of the form \eqref{ya} but contains $\ln (z\mp p)$ terms. The singularity $\pm p$ is called \emph{apparent} if any solution of GLE$(\mathbf{n},p,A,\tau)$ has no logarithmic singularity at $\pm p$, i.e. GLE$(\mathbf{n},p,A,\tau)$ also has a local solution $y_{\kappa}(z)$ with $\kappa=-\frac12$. It was proved in \cite{Chen-Kuo-Lin0} that $\pm p$ are both apparent singularities if and only if \eqref{101} holds.
\item For $a=\frac{\omega_k}{2}$,  since the potential $I_{\mathbf{n}}(z;p,A,\tau)$ is even elliptic as a function of $z$, it is standard to see (cf. \cite{Chen-Kuo-Lin0,CKL1,Takemura}) that  GLE$(\mathbf{n},p,A,\tau)$ always has a local solution $y_{\kappa}(z)$ with each $\kappa\in \{-n_k, n_k+1\}$, i.e. $\frac{\omega_k}{2}$ is always apparent. Since $n_k\in \mathbb{N}$, it follows that all solutions of GLE$(\mathbf{n},p,A,\tau)$ are single-valued and meromorphic near the singularity $\frac{\omega_k}{2}$.
\end{itemize}
\end{remark}

Fix any base point $q_{0}\in E_{\tau}$ that is not a singularity of \eqref{89-1}. The monodromy representation of GLE$(\mathbf{n},p,A,\tau)$ is a group homomorphism $\rho:\pi_{1}(  E_{\tau}\backslash
(\{  \pm [p]\}  \cup E_{\tau}[2]),q_{0})  \rightarrow
SL(2,\mathbb{C})$ defined as follows. Take any basis of solutions $(y_1(z), y_2(z))$ of GLE$(\mathbf{n},p,A,\tau)$. For any loop $\gamma\in \pi_{1}(  E_{\tau}\backslash
(\{  \pm [p]\}  \cup E_{\tau}[2]),q_{0})$, let $\gamma^*y(z)$ denote the analytic continuation of $y(z)$ along $\gamma$. Then $\gamma^*(y_1(z), y_2(z))$ is also a basis of solutions, so there is a matrix $\rho(\gamma)\in SL(2,\mathbb{C})$ such that
$$\gamma^*\begin{pmatrix}
y_{1}(z)\\
y_{2}(z)
\end{pmatrix}
=\rho(\gamma)
\begin{pmatrix}
y_{1}(z)\\
y_{2}(z)
\end{pmatrix}.$$
Here $\rho(\gamma)\in SL(2,\mathbb{C})$ (i.e. $\det \rho(\gamma)=1$) follows from the fact that the Wronskian $y_{1}'(z)y_2(z)-y_1(z)y_2'(z)$ is a nonzero constant.

Since Remark \ref{rmk2-0} says that all solutions of GLE$(\mathbf{n},p,A,\tau)$ are single-valued and meromorphic near the singularity $\frac{\omega_k}{2}$, we know that  the
local monodromy matrix at $\frac{\omega_k}{2}$ is always the identity matrix $I_{2}$. Thus the
monodromy representation can be reduced to $\rho:\pi
_{1}(  E_{\tau}\backslash \{  \pm [p]\}  ,q_{0})
\rightarrow SL(2,\mathbb{C})$. Let $\gamma_{\pm}\in \pi_{1}(  E_{\tau
}\backslash(\{ \pm [p]\} \cup E_{\tau}[2]),q_{0})  $ be a simple
loop encircling $\pm p$ counterclockwise respectively, and $\ell_{j}\in \pi
_{1}(  E_{\tau}\backslash(\{ \pm p\} \cup E_{\tau}[2]),q_{0})$, $j=1,2$, be two fundamental cycles of $E_{\tau}$ connecting
$q_{0}$ with $q_{0}+\omega_{j}$ such that $\ell_{j}$ does not intersect with
$L+\Lambda_{\tau}$ (here $L$ is the straight segment connecting $\pm p$) and
satisfies%
\[
\gamma_{+}\gamma_{-}=\ell_{1}\ell_{2}\ell_{1}^{-1}\ell_{2}^{-1}\text{ in }%
\pi_{1}(  E_{\tau}\backslash  \{  \pm [p] \}
,q_{0})  .
\]
Since the local exponents of (\ref{89-1}) at $\pm p$ are $
\{-\frac{1}{2}, \frac{3}{2}\}$ and $\pm p\not \in E_{\tau}[2]$ are apparent singularities, we
always have
$\rho(\gamma_{\pm})=-I_{2}$.
Denote by $N_{j}=\rho(\ell_j)$ the monodromy matrix along the loop $\ell_{j}$ of
(\ref{89-1}) with respect to any basis of solutions. Then $N_1N_2=N_2N_1$ and the monodromy group of (\ref{89-1}) is
generated by $\{-I_{2},N_{1},N_{2}\}$, i.e. is always \emph{abelian and so reducible}.
Then there are two cases (see \cite{CKL1,Chen-Kuo-Lin}):

\begin{itemize}
\item[(a)] Completely reducible (i.e. all the monodromy matrices have two
linearly independent common eigenfunctions, or equivalently, $N_1$ and $N_2$ can be diagonalized simultaneously). Up to a common conjugation,
$N_1$ and $N_2$ can be expressed as%
\begin{equation}\label{n1n2}N_{1}=%
\begin{pmatrix}
e^{-2\pi is} & 0\\
0 & e^{2\pi is}%
\end{pmatrix},\quad N_{2}=%
\begin{pmatrix}
e^{2\pi ir} & 0\\
0 & e^{-2\pi ir}%
\end{pmatrix}\end{equation}
for some $(r,s)\in \mathbb{C}^{2}\backslash \frac{1}{2}\mathbb{Z}^{2}$. 

\item[(b)] Not completely reducible (i.e. the space of common eigenfunctions
is of dimension $1$, or equivalently, $N_1$ and $N_2$ can not be diagonalized simultaneously). Up to a common conjugation, $N_1$ and
$N_2$ can be expressed as%
\begin{equation}
N_1=\varepsilon_{1}%
\begin{pmatrix}
1 & 0\\
1 & 1
\end{pmatrix}
,\text{ \  \  \ }N_2=\varepsilon_{2}%
\begin{pmatrix}
1 & 0\\
\mathcal{C} & 1
\end{pmatrix}
, \label{Mono-21}%
\end{equation}
where $\varepsilon_{1},\varepsilon_{2}\in \{ \pm1\}$ and $\mathcal{C}%
\in \mathbb{C}\cup \{ \infty \}$.
Here if $\mathcal{C}=\infty$, then (\ref{Mono-21}) should be understood
as%
\[
N_1=\varepsilon_{1}%
\begin{pmatrix}
1 & 0\\
0 & 1
\end{pmatrix}
,\text{ \  \  \ }N_2=\varepsilon_{2}%
\begin{pmatrix}
1 & 0\\
1 & 1
\end{pmatrix}
. \]

\end{itemize}

\begin{remark}\label{rmk2-1}
We say the monodromy of GLE$(\mathbf{n},p,A,\tau)$ is unitary if up to a common conjugation, the monodromy group of GLE$(\mathbf{n},p,A,\tau)$ is a subgroup of the unitary group $SU(2)$, which is equivalent to $N_1, N_2\in SU(2)$ up to a common conjugation. Clearly the monodromy of GLE$(\mathbf{n},p,A,\tau)$ is unitary if and only if Case (a) happens with $(r,s)\in\mathbb{R}^2\setminus\frac12\mathbb{Z}^2$.
\end{remark}

The deep connection between GLE$(\mathbf{n},p,A,\tau)$ and the elliptic form of Painlev\'{e} VI equation EPVI$_{\mathbf{n}}$ was studied in \cite{Chen-Kuo-Lin0}.

\begin{theorem}\cite[Theorems 1.1-1.3]{Chen-Kuo-Lin0}\label{thm-2-6}
Let
$U$ be an open subset of $\mathbb{H}$ such that $p(\tau)\not \in E_{\tau}[2]$
for any $\tau \in U$. Then
$p(\tau)$ {is a solution of EPVI$_{\mathbf{n}}$ if and only
if there exist $A(\tau)$ (and the corresponding $B(\tau)$ via \eqref{101})
 such that GLE$(\mathbf{n},p(\tau),A(\tau),\tau)$ is monodromy preserving
as $\tau \in U$ deforms}. Furthermore, $(p(\tau),A(\tau))$ satisfies the following Hamiltonian system%
\begin{equation}
\left \{
\begin{array}
[c]{l}%
\frac{dp(\tau)}{d\tau}=\frac{\partial \mathcal{H}}{\partial A}=\frac{-i}{4\pi
}(2A-\zeta(2p;\tau)+2p\eta_{1}(\tau))\\
\frac{dA(\tau)}{d\tau}=-\frac{\partial \mathcal{H}}{\partial p}=\frac{i}{4\pi
}\left(
\begin{array}
[c]{l}%
(2\wp(2p;\tau)+2\eta_{1}(\tau))A-\frac{3}{2}\wp^{\prime}(2p;\tau)\\
-\sum_{k=0}^{3}n_{k}(n_{k}+1)\wp^{\prime}(p-\frac{\omega_{k}}{2};\tau)
\end{array}
\right)
\end{array}
\right.  . \label{142-0}%
\end{equation}
where
$
\mathcal{H}  =\frac{-i}{4\pi}(B+2p\eta_{1}(\tau)A)$. In other words, this Hamiltonian system \eqref{142-0} is equivalent to EPVI$_{\mathbf{n}}$.
\end{theorem}

Recalling the solution $p_{r,s}^{\mathbf{n}}(\tau)$ of EPVI$_{\mathbf{n}}$ defined in Lemma B, 
the following result answers the natural question that when the monodromy of GLE$(\mathbf{n},p(\tau),A(\tau),\tau)$ satisifes \eqref{n1n2}.

\begin{theorem}\cite[Theorem 5.3]{Chen-Kuo-Lin} \label{thm-2-7}
For $\mathbf{n}$, let $p^{\mathbf{n}}(\tau)$ be a
solution to EPVI$_{\mathbf{n}}$ and $A^{\mathbf{n}}(\tau)$ be defined by the first equation of the Hamiltonian system \eqref{142-0}  (and the corresponding $B(\tau)$ via \eqref{101}). Then for any $\tau$ satisfying $p^{\mathbf{n}}
(\tau)\not \in E_{\tau}[2]$, the monodromy group of the
associated GLE$(\mathbf{n}$, $ p^{\mathbf{n}}(\tau), A^{\mathbf{n}}(\tau), \tau)$ is generated by
\begin{equation}
\rho(\gamma_{\pm})=-I_{2},\quad N_{1}=%
\begin{pmatrix}
e^{-2\pi is} & 0\\
0 & e^{2\pi is}%
\end{pmatrix}
\quad N_{2}=%
\begin{pmatrix}
e^{2\pi ir} & 0\\
0 & e^{-2\pi ir}%
\end{pmatrix}
 \label{II-101}%
\end{equation}
if and
only if $(r,s)  \in \mathbb{C}^{2}\backslash \frac
{1}{2}\mathbb{Z}^{2}$ and $p^{\mathbf{n}}(\tau)=p_{r,s}^{\mathbf{n}
}(\tau)$ in the sense of Remark \ref{identify}. 
\end{theorem}

Theorem \ref{thm-2-7} indicates that for the solution $p_{r,s}^{\mathbf{n}}(\tau)$ of EPVI$_{\mathbf{n}}$, the $(r,s)$ represents precisely the mondromy data of the associated GLE. 
Theorems \ref{thm-2-6} and \ref{thm-2-7} will play crucial roles in our proof of Theorem \ref{theorem-1} in the next section.

\section{Proof of Theorem \ref{theorem-1}}

In this section, we establish the deep connection between the curvature equation \eqref{mfe-p} and GLE$(\mathbf{n},p,A,\tau)$, and give the proof of Theorem \ref{theorem-1}. The key step is to prove the following result.

\begin{theorem}\label{thm-3-1}
Suppose $n_{k}\in\N$ for all $k$ and given
$\tau \in \mathbb{H}$, and $p\not\in E_{\tau}[2]$. 
Then
the curvature equation \eqref{mfe-p}
$$
\Delta u+e^{u}=8\pi \sum_{k=0}^{3}n_{k}\delta_{\frac{\omega_{k}}{2}}+4\pi \left(  \delta_{p}+\delta_{-p}\right)  \quad\text{ on }\;
E_{\tau}$$
has even solutions (the assumption``even'' is not needed for the special case $\mathbf n=\mathbf 0$) if and only if there exists $A\in \mathbb{C}$ (and hence $B$ via \eqref{101}) such
that the monodromy of  GLE$(\mathbf{n},p,A,\tau)$ is unitary.
\end{theorem}

\begin{proof}
{\bf Step 1.} We prove the necessary part.
Let $u(z)$ be a solution of (\ref{mfe-p}). Then
the classical Liouville theorem (cf. \cite[Section 1.1]{CLW}) says that there
is a local meromorphic function $f(z)$ away from $E_{\tau}[2]\cup\{\pm [p]\}$ such that%
\begin{equation}
u(z)=\log \frac{8|f^{\prime}(z)|^{2}}{(1+|f(z)|^{2})^{2}}. \label{502}%
\end{equation}
This $f(z)$ is called a developing map. By differentiating (\ref{502}), we
have
\begin{equation}
u_{zz}-\frac{1}{2}u_{z}^{2}= \{ f;z \}:=\left(  \frac{f^{\prime \prime}%
}{f^{\prime}}\right)  ^{\prime}-\frac{1}{2}\left(  \frac{f^{\prime \prime}%
}{f^{\prime}}\right)  ^{2}. \label{new22}%
\end{equation}
Conventionally, the RHS of this identity is called the Schwarzian derivative
of $f(z)$, denoted by $\{ f;z \}$. Note that outside
the singularities $E_{\tau}[2]\cup\{\pm [p]\}$,
\[
\left(  u_{zz}-\tfrac{1}{2}u_{z}^{2}\right)  _{\bar{z}}=\left(  u_{z\bar{z}%
}\right)  _{z}-u_{z}u_{z\bar{z}}=-\tfrac{1}{4}\left(  e^{u}\right)
_{z}+\tfrac{1}{4}e^{u}u_{z}=0.
\]
On the other hand, since $v(z):=u(z)-4n_k\ln|z-\tfrac{\omega_k}{2}|$ solves $\Delta v+|z-\frac{\omega_k}{2}|^{4n_k}e^v=0$ near $\frac{\omega_k}{2}$, it is standard (cf. \cite{B-M}) to see that $v(z)$ is regular at $\frac{\omega_k}{2}$, i.e. $$u(z)=4n_k\ln|z-\frac{\omega_k}{2}|+O(1)\quad\text{near }\;\frac{\omega_k}{2},$$
and then
$$u_{zz}-\frac{1}{2}u_{z}^{2}=\frac{-2n_{k}(n_{k}+1)}{(z-\frac{\omega_k}{2})^2}+O(\frac{1}{z-\frac{\omega_k}{2}})\quad\text{near }\;\frac{\omega_k}{2}.$$
Similarly,
$$u(z)=2\ln|z\pm p|+O(1)\quad\text{near }\;\mp p,$$
and
$$u_{zz}-\frac{1}{2}u_{z}^{2}=\frac{-3}{2(z\pm p)^2}+O(\frac{1}{z\pm p})\quad\text{near }\;\mp p.$$
 From here
we conclude that $u_{zz}-\frac{1}{2}u_{z}^{2}$ is an \emph{elliptic function} with at most \emph{double poles} at $E_\tau[2]\cup\{\pm[p]\}$.

Now we further assume that $u(z)$ is even (i.e. $u(z)=u(-z)$) for the general case $\mathbf n\neq \mathbf 0$. Then the (even) elliptic function $u_{zz}-\frac{1}
{2}u_{z}^{2}$ has no residues at $z\in E_{\tau}[2]$, i.e.
$$\underset{z=\frac{\omega_{k}}{2}}{\text{Res}}(u_{zz}-\frac{1}
{2}u_{z}^{2})=0,\quad\forall k=0,1,2,3.$$
Furthermore, since the sum of residues of any elliptic function is zero,  there is a constant $A=A(u)$ such that
\[
2A=\underset{z=p}{\text{Res}}(  u_{zz}-\frac{1}{2}u_{z}%
^{2})  =-\underset{z=-p}{\text{Res}}(  u_{zz}-\frac{1}{2}u_{z}%
^{2}),
\]
so the elliptic function
\begin{align*}u_{zz}-\frac{1}{2}u_{z}^{2}+2\Big[
&\sum_{k=0}^{3}n_{k}(n_{k}+1)\wp(z-\tfrac{\omega_{k}}{2};\tau)+\frac{3}{4}
(\wp(z+p;\tau)+\wp(z-p;\tau))\nonumber\\
&+A(\zeta(z+p;\tau)-\zeta(z-p;\tau))
\Big]\end{align*}
has no poles and hence must be a constant, which implies the existence of a constant $B=B(u)$ such that 
\begin{align}
u_{zz}-\frac{1}{2}u_{z}^{2}=-2\Big[
&\sum_{k=0}^{3}n_{k}(n_{k}+1)\wp(z-\tfrac{\omega_{k}}{2};\tau)+\frac{3}{4}
(\wp(z+p;\tau)+\wp(z-p;\tau))\nonumber\\
&+A(\zeta(z+p;\tau)-\zeta(z-p;\tau))+B
\Big]=-2I_{\mathbf{n}}(z;p,A,\tau). \label{cc-1}%
\end{align}
We will explain below that $B$ is actually given by \eqref{101}.

Since the  Schwarzian derivative $\{f;z\}=-2I_{\mathbf{n}}(z;p,A,\tau)$, a classical result in complex analysis says that there are linearly independent solutions $y_1(z), y_2(z)$ of GLE$(\mathbf{n},p,A,\tau)$ such that
\begin{equation}\label{fy}f(z)=\frac{y_1(z)}{y_2(z)}.\end{equation}
Define the Wronskian
\[W:=y_{1}'(z)y_2(z)-y_1(z)y_2'(z).\]
Then $W$ is a nonzero constant. By inserting (\ref{fy}) into (\ref{502}), a direct computation leads to
\[2\sqrt{2}W e^{-\frac{1}{2}u(z)}=|y_1(z)|^2+|y_2(z)|^2.\]
Note that as a solution of the curvature equation \eqref{mfe-p}, $u(z)$ is single-valued and doubly periodic, so $e^{-\frac{1}{2}u(z)}$ is invariant under analytic continuation along any loop $ \gamma\in\pi
_{1}(  E_{\tau}\backslash \{  \pm [p]\}  ,q_{0})$. This implies that for any monodromy matrix $\bigl(\begin{smallmatrix}a & b\\
c & d\end{smallmatrix}\bigr)\in SL(2,\mathbb{C})$ with respect to $(y_1(z), y_2(z))$, there holds
\[|ay_1(z)+by_2(z)|^2+|cy_1(z)+dy_2(z)|^2=|y_1(z)|^2+|y_2(z)|^2,\]
which easily infers $\bigl(\begin{smallmatrix}a & b\\
c & d\end{smallmatrix}\bigr)\in SU(2)$. Therefore, the monodromy of GLE$(\mathbf{n},p,A,\tau)$ is unitary. Remark that if GLE$(\mathbf{n},p,A,\tau)$ has logarithmic singularities at $\pm p$, then it is well-known that the local monodromy matrix at $\pm p$ can not be in $SU(2)$. Thus, GLE$(\mathbf{n},p,A,\tau)$ is apparent at $\pm p$, which implies that $B$ is given by \eqref{101}. 

{\bf Step 2.}
We prove the sufficient part.
Suppose  there exists $A\in \mathbb{C}$ (and $B$ via \eqref{101}) such
that the monodromy of  GLE$(\mathbf{n},p,A,\tau)$ is unitary. Then by Remark \ref{rmk2-1}, there exist a basis of solutions $(y_1(z), y_2(z))$ of  GLE$(\mathbf{n},p,A,\tau)$ and $(r,s)\in\mathbb{R}^2\setminus\frac{1}{2}\mathbb{Z}^2$ such that  the monodromy matrices $N_1, N_2$ with respect to $(y_1(z), y_2(z))$ are given by \eqref{n1n2}, namely
\begin{equation}
\ell_1^*\begin{pmatrix}
y_{1}(z)\\
y_{2}(z)
\end{pmatrix}
=\left(
\begin{matrix}
e^{-2\pi is} & 0\\
0 & e^{2\pi is}%
\end{matrix}
\right)
\begin{pmatrix}
y_{1}(z)\\
y_{2}(z)
\end{pmatrix}
, \label{61-34}%
\end{equation}%
\[
\ell_2^*\begin{pmatrix}
y_{1}(z)\\
y_{2}(z)
\end{pmatrix}
=\left(
\begin{matrix}
e^{2\pi ir} & 0\\
0 & e^{-2\pi ir}%
\end{matrix}
\right)
\begin{pmatrix}
y_{1}(z)\\
y_{2}(z)
\end{pmatrix}
.
\]
Note that if $y(z)$ is a solution of GLE$(\mathbf{n},p,A,\tau)$, then so is $y_1(-z)$.
It follows from \cite[Remark 3.2]{CKL1} that we can take $y_2(z)=y_1(-z)$.

Now we define
\begin{equation}\label{deve-f}f(z):=\frac{y_{1}(z)}{y_{2}(z)}=\frac{y_1(z)}{y_1(-z)},\end{equation}
and
\begin{equation}\label{deve-fff}u(z):=\log \frac{8|f^{\prime}(z)|^{2}}{(1+|f(z)|^{2})^{2}}.\end{equation}
We claim that this $u(z)$ is an even solution of (\ref{mfe-p}).

First, since the local exponents of $y_j$ at $\pm p$ both belong to $\{-\frac12, \frac32\}$, so the local exponents of $f$ at $\pm p$ must belong to $\{0, \pm 2\}$, which implies that $f$ is single-valued near $\pm p$ (although $y_j$ is multi-valued near $\pm p$). Therefore, after analytic continuation $f$ can be extended to be a meromorphic function in $\mathbb{C}$. Then \eqref{61-34} and \eqref{deve-f} imply that
$$f(z+\omega_1)=\ell_1^*f(z)=e^{-4\pi is}f(z),\quad f(z+\omega_2)=\ell_2^*f(z)=e^{4\pi ir}f(z).$$
From here and $(r,s)\in\mathbb{R}^2$, we see that $u(z)$ defined by \eqref{deve-fff} is doubly periodic and hence well-defined on $E_{\tau}$. Furthermore, it follows from \eqref{deve-f} that $f(-z)=\frac{1}{f(z)}$, which infers that $u(z)=u(-z)$. 

Notice that by \eqref{deve-f}, a direct computation gives
\begin{align}\label{swaz}\left(  \frac{f^{\prime \prime}%
}{f^{\prime}}\right)  ^{\prime}-&\frac{1}{2}\left(  \frac{f^{\prime \prime}%
}{f^{\prime}}\right)  ^{2}=-2I_{\mathbf{n}}(z;p,A,\tau)=-2\Big[
\sum_{k=0}^{3}n_{k}(n_{k}+1)\wp(z-\tfrac{\omega_{k}}{2};\tau)\nonumber\\
&+\frac{3}{4}
(\wp(z+p;\tau)+\wp(z-p;\tau))+A(\zeta(z+p;\tau)-\zeta(z-p;\tau))+B
\Big].\end{align}

Fix any $a\in E_{\tau}\setminus(E_{\tau}[2]\cup\{\pm[p]\})$, we claim that $u(a)\neq \infty$.
In fact, since the local exponent of $y_j$ at $a$ belongs to $\{0,1\}$, so the local exponent of $f$ at $a$ must belong to $\{0, \pm 1\}$. If the local exponent is $\pm 1$, i.e.
$f(z)\sim (z-a)^{\pm 1}$ at $a$, then it follows from \eqref{deve-fff} that $u(a)\neq \infty$. If the local exponent is $0$, i.e. $f(z)=f(a)+c(z-a)^m+O((z-a)^{m+1})$ for some $m\geq 1$ and $c\neq 0$, then by inserting this formula into \eqref{swaz}, we easily obtain $m=1$ and so \eqref{deve-fff} implies $u(a)\neq \infty$.
Consequently, a direct computation via \eqref{deve-fff} implies
\[\Delta u+e^u=0\quad \text{on}\; E_{\tau}\setminus(E_{\tau}[2]\cup\{\pm[p]\}).\]

Recall that the local exponent of $f$ at $a\in\{\pm p\}$ must belong to $\{0, \pm 2\}$. If the local exponent is $\pm 2$, i.e.
$f(z)\sim (z-a)^{\pm 2}$ at $a\in\{\pm p\}$, then it follows from \eqref{deve-fff} that \begin{equation}\label{fcfc}u(z)=2\ln|z-a|+O(1)\quad\text{near }\;a\in\{\pm p\}.\end{equation} If the local exponent of $f$ at $a\in\{\pm p\}$ is $0$, i.e. $f(z)=f(a)+c(z-a)^m+O((z-a)^{m+1})$ for some $m\geq 1$ and $c\neq 0$, then by inserting this formula into \eqref{swaz}, we easily obtain $m=2$ and so \eqref{deve-fff} implies \eqref{fcfc} again.

Similarly, since the local exponent of $y_j$ at $\frac{\omega_k}{2}$ both belong to $\{-n_k, n_k+1\}$, so the local exponent of $f$ at $\frac{\omega_k}{2}$ must belong to $\{0, \pm (2n_k+1)\}$.  If the local exponent is $\pm (2n_k+1)$, i.e.
$f(z)\sim (z-\frac{\omega_k}{2})^{\pm (2n_k+1)}$ at $\frac{\omega_k}{2}$, then it follows from \eqref{deve-fff} that \begin{equation}\label{fcfc1}u(z)=4n_k\ln|z-\frac{\omega_k}{2}|+O(1)\quad\text{near }\;\frac{\omega_k}{2}.\end{equation} If the local exponent of $f$ at $\frac{\omega_k}{2}$ is $0$, i.e. $f(z)=f(\frac{\omega_k}{2})+c(z-\frac{\omega_k}{2})^m+O((z-\frac{\omega_k}{2})^{m+1})$ for some $m\geq 1$ and $c\neq 0$, then by inserting this formula into \eqref{swaz}, we easily obtain $m=2n_k+1$ and so \eqref{deve-fff} implies \eqref{fcfc1} again.

In conclusion, the above argument shows that $u(z)$ is an even solution of the curvature equation (\ref{mfe-p}).
The proof is complete.
\end{proof}

\begin{remark}
Let us turn back to Step 2 of the proof of Theorem \ref{thm-3-1}. Like \eqref{deve-fff}, we define
\begin{equation}\label{deve-ffff}u_\beta(z):=\log \frac{8|\beta f^{\prime}(z)|^{2}}{(1+|\beta f(z)|^{2})^{2}},\quad\forall\beta>0,\end{equation}
namely we replace $f(z)$ with $\beta f(z)$ in \eqref{deve-fff}. Then exactly the same argument as Step 2 above shows that $u_\beta(z)$ is a solution of the curvature equation \eqref{mfe-p} for any $\beta>0$. The only different thing for $\beta\neq 1$ is that, since $\beta f(-z)\neq \frac{1}{\beta f(z)}$ for $\beta\neq 1$, we see that $u_\beta(z)$ is not even for $\beta\neq 1$. Therefore, we obtain
\end{remark}

\begin{corollary}\label{coro-3-3}
Suppose $n_{k}\in\N$ for all $k$ and given
$\tau \in \mathbb{H}$, and $p\not\in E_{\tau}[2]$. 
Suppose
the curvature equation \eqref{mfe-p}
has an even solution (the assumption``even'' is not needed for the special case $\mathbf n=\mathbf 0$). Then \eqref{mfe-p} has a 1-parameter family of solutions $u_{\beta}(z)$, where $\beta>0$ is arbitrary. Furthermore, $u_{\beta}(z)$ is even if and only if $\beta=1$.
\end{corollary}

The same statement as Corollary \ref{coro-3-3} was proved firstly in \cite{CLW} for the curvature equation with a single singularity $$\Delta u+e^u=8\pi n\delta_0\quad\text{on }E_{\tau},\quad n\in\mathbb{N}_{\geq 1}. $$

\begin{remark}\label{rmk3-4}
It is easy to see from \eqref{deve-ffff} that $\lim\limits_{\beta\to+\infty}u_\beta(z)=+\infty$ if $z$ is a simple zero of $f$, and similarly, $\lim\limits_{\beta\to0}u_\beta(z)=+\infty$ if $z$ is a simple pole of $f$, namely the solution sequence $u_\beta$ blows up as $\beta\to +\infty$ or $\beta\to 0$. An interesting problem is whether $u_\beta$ might blow up at some singularities (say $\pm p$ for example). We should study this problem in future.
\end{remark}

Now we are ready to prove Theorem \ref{theorem-1}.

\begin{proof}[Proof of Theorem \ref{theorem-1}]
First we prove the necessary part. Given
$\tau_0 \in \mathbb{H}$ and $p\not\in E_{\tau_0}[2]$. 
Suppose the curvature equation
\begin{equation}\label{mfe-p0}
\Delta u+e^{u}=8\pi \sum_{k=0}^{3}n_{k}\delta_{\frac{\omega_{k}}{2}}+4\pi \left(  \delta_{p}+\delta_{-p}\right)  \quad\text{ on }\;
E_{\tau_0}\end{equation}
has an even
solution $u(z)$ (the assumption ``even'' is not needed for the case $\mathbf n=\mathbf 0$). Then by Theorem \ref{thm-3-1}, there exists $A\in \mathbb{C}$ (and hence $B$ via \eqref{101}) such
that the monodromy of  GLE$(\mathbf{n},p,A,\tau_0)$
\begin{align} \label{89-11}
y''=\Big[
&\sum_{k=0}^{3}n_{k}(n_{k}+1)\wp(z-\tfrac{\omega_{k}}{2};\tau_0)+\frac{3}{4}
(\wp(z+p;\tau_0)+\wp(z-p;\tau_0))\\
&+A(\zeta(z+p;\tau_0)-\zeta(z-p;\tau_0))+B
\Big]  y\nonumber
\end{align}
is unitary. In particular, it follows from Remark \ref{rmk2-1} that there exists $(r,s)\in\mathbb{R}^2\setminus\frac{1}{2}\mathbb{Z}^2$ such that the monodromy matrices $N_1, N_2$ of GLE$(\mathbf{n},p,A,\tau_0)$ are given by \eqref{n1n2}.
On the other hand, the initial value problem for the
Hamiltonian system (\ref{142-0}) with%
\[
p(\tau_{0})=p,\text{ \ }A(\tau_{0})=A
\]
has a solution $(p(\tau),A(\tau))$. Applying Theorem \ref{thm-2-6}, we see that $p(\tau)$
is a solution of EPVI$_{\mathbf n}$.
Furthermore, by defining $B=B(\tau)$ via (\ref{101}) with $(p(\tau),A(\tau))$,
the corresponding GLE$(\mathbf{n},p(\tau),A(\tau),\tau)$ is
monodromy preserving and coincides with GLE$(\mathbf{n},p,A,\tau_0)$ at $\tau=\tau_{0}$. Thus,
the monodromy matrices $N_1, N_2$ of GLE$(\mathbf{n},p(\tau)$, $A(\tau),\tau)$ are also given by \eqref{n1n2}. From here and Theorem \ref{thm-2-7}, we conclude that $p(\tau)=p_{r,s}^{\mathbf{n}
}(\tau)$ in the sense of Remark \ref{identify}. In particular, $p=p(\tau_0)=\pm p_{r,s}^{\mathbf{n}
}(\tau_0)\in \Omega_{\tau_0}^{\mathbf n}\setminus E_{\tau_0}[2]$. This proves the
necessary part.

Now we prove the sufficient part. Suppose $p\in \Omega_{\tau_0}^{\mathbf n}\setminus E_{\tau_0}[2]$, namely there is $(r,s)\in\mathbb{R}^2\setminus\frac{1}{2}\mathbb{Z}^2$ such that $p=\pm p_{r,s}^{\mathbf{n}
}(\tau_0)$. By replacing $p$ with $-p$ if necessary, we may assume $p=p_{r,s}^{\mathbf{n}
}(\tau_0)$. By defining $A=A(\tau_0)$ via the first equation of the
Hamiltonian system (\ref{142-0}) at $\tau=\tau_0$ (and hence $B$ via \eqref{101} with $\tau=\tau_0$), it follows from Theorem \ref{thm-2-7} that the monodromy matrices $N_1, N_2$ of GLE$(\mathbf{n},p$, $A,\tau_0)$ are given by \eqref{n1n2}, namely the monodromy of GLE$(\mathbf{n},p$, $A,\tau_0)$ is unitary. Then we conclude from Theorem \ref{thm-3-1} that the curvature equation \eqref{mfe-p0} has an even solution.
The proof is complete.
\end{proof}

\section{Proof of Theorems \ref{main-thm2} and\ref{thm=4-3}}

In this section, we consider $\mathbf n=\mathbf 0$ and prove Theorems \ref{main-thm2} and \ref{thm=4-3}. 
First, we recall the following result for the Green function $G(z)=G(z;\tau)$ when $\tau\in i\mathbb{R}_{>0}$.

\begin{theorem}\cite{LW}\label{thm-LW}
For $\tau\in i\mathbb{R}_{>0}$, i.e. $E_{\tau}$ is a rectangular torus, the Green function $G(z)$ has exactly $3$ critical points $\frac{\omega_k}{2}$, $k=1,2,3$. Furthermore, all the $3$ critical points are non-degenerate.
\end{theorem}

The Green function $G(z)$ on $E_{\tau}$ can be expressed explicitly in terms of elliptic
functions. In particular, by using the complex variable $z\in\mathbb C$, it was proved in \cite{LW} that
\begin{equation}
\label{G_z}-4\pi \frac{\partial G}{\partial z}(z;\tau)=\zeta(z;\tau)-r\eta_{1}(\tau)
-s\eta_{2}(\tau),
\end{equation}
where $(r,s)$ is defined by  $$z=r+s\tau\quad\text{ with }\quad r,s\in \mathbb{R}.$$
Note that $z\notin E_{\tau}[2]$ is equivalent to $(r,s)\in\mathbb{R}^2\setminus\frac12\mathbb{Z}^2$.
Using this formula and Hitchin's formula \eqref{513-1}, we can prove Theorem \ref{thm=4-3}.

\begin{theorem}[=Theorem \ref{thm=4-3}]\label{thm4-1}
Given $p\in E_{\tau}\setminus E_{\tau}[2]$ and $(r,s)\in\mathbb{R}^2\setminus\frac12\mathbb{Z}^2$. Then $a:=r+s\tau$ is a nontrivial critical point of $G_p$ if and only if
\begin{equation}\label{a+p1}\wp(p;\tau)=\wp(r+s\tau;\tau)+\frac{\wp^{\prime
}(r+s\tau;\tau)}{2Z_{r,s}(\tau)  }.\end{equation}
Consequently, $G_p(z)$ has nontrivial critical points if and only if $p\in \Omega_{\tau}^{\mathbf 0}$, if and only if
$$\Delta u+e^u=4\pi(\delta_p+\delta_{-p})\quad\text{on }E_{\tau}$$
has solutions.
\end{theorem}

\begin{proof} Note from $(r,s)\in\mathbb{R}^2\setminus\frac12\mathbb{Z}^2$ that $a\notin E_{\tau}[2]$.
By \eqref{G_z}, we see that $a$ is a critical point of $G_p$ if and only if 
\begin{equation}\label{a+1p}\zeta(a+p)+\zeta(a-p)-2(r\eta_1+s\eta_2)=0.\end{equation}
Applying the addition formula of elliptic functions
$$\zeta(a+p)+\zeta(a-p)-2\zeta(a)=\frac{\wp'(a)}{\wp(a)-\wp(p)},$$
it is easy to see that \eqref{a+1p} is equivalent to \eqref{a+p1}. The rest assertions follows directly from Theorem \ref{theorem-1}.
\end{proof}

Now we are ready to prove Theorem \ref{main-thm2}.

\begin{proof}[Proof of Theorem \ref{main-thm2}]
Let $\tau\in i\mathbb{R}_{>0}$. First, we prove that there is $\varepsilon>0$ such that $\Omega_{\tau}^{\mathbf 0}\cap B(0, \varepsilon)=\emptyset$.
Assume by contradition that there exist $p_n\in\Omega_{\tau}^{\mathbf 0}$ such that $p_n\to 0$ as $n\to+\infty$. Then by Theorem \ref{thm4-1}, $G_{p_n}$ has a nontrivial critical point $a_n=r_n+s_n\tau\in E_\tau\setminus E_{\tau}[2]$, where $(r_n, s_n)\in\mathbb{R}^2\setminus\frac12\mathbb{Z}^2$. Up to a subsequence, we may assume $a_n\to a_{\infty}\in E_{\tau}$.
We claim that $a_{\infty}\neq 0$ in $E_{\tau}$.

Assume by contradiction that $a_{\infty}=0$ in $E_{\tau}$, by replacing $a_n$ with some element from $a_n+\mathbb{Z}+\mathbb{Z}\tau$, we may assume $a_n\to 0$ and so $r_n, s_n\to 0$. Then 
 it follows from the Laurent series of $\zeta(z)$ and $\wp(z)$ that
\[
\zeta(a_n)=\frac{1}{a_n}+O(|a_n|^{3}),\quad
\wp(a_n)=\frac{1}{a_n^{2}}+O(|a_n|^{2}),
\]%
\[
\wp^{\prime}(a_n)=\frac{-2}{a_n^{3}}+O(|a_n|).
\]
Inserting these into \eqref{a+p1}, we obtain
\begin{align*}
\wp(p_n)=\wp(a_n)+\frac{\wp^{\prime
}(a_n)}{2(\zeta(a_n)-r_n\eta_1-s_n\eta_2) }=-\frac{r_n\eta_1+s_n\eta_2}{a_n}+O(|a_n|^{2}).
\end{align*}
Note from $\operatorname{Im}\tau>0$ and $\mathbb{R}^2\ni (r_n, s_n)\to (0,0)$ that $\frac{r_n\eta_1+s_n\eta_2}{a_n}=\frac{r_n\eta_1+s_n\eta_2}{r_n+s_n\tau}$ is uniformly bounded, we obtain a contradiction with $\wp(p_n)\to \infty$ as $p_n\to 0$. 

Therefore, $a_\infty\neq 0$ in $E_{\tau}$. Then $a_\infty$ is a critical point of $G(z)$. Since $\tau\in i\mathbb{R}_{>0}$, it follows from Theorem \ref{thm-LW} that $a_\infty=\frac{\omega_k}{2}$ for some $k=1,2,3$. Therefore, $G_{p_n}(z)$ has two different critical points $a_n$ and $\frac{\omega_k}{2}$ that both converge to $\frac{\omega_k}{2}$ as $n\to\infty$. However, since Theorem \ref{thm-LW} says that $\frac{\omega_k}{2}$ is a non-degenerate critical point of $G(z)$, we conclude from the implicit function theorem that there exists $\delta>0$ small such that $G_p(z)$ has exactly one critical point (hence must be $\frac{\omega_k}{2}$) in a small neighborhood of $\frac{\omega_k}{2}$ for any $|p|<\delta$. Since $|p_n|<\delta$ for $n$ large, we obtain a contradiction.

Therefore, there exists $\varepsilon>0$ small such that $\Omega_{\tau}^{\mathbf 0}\cap B(0,\varepsilon)=\emptyset$. Finally, since $-(p-\frac{\omega_k}{2})=-p-\frac{\omega_k}{2}$ in $E_{\tau}$, we see that $u(z)$ is a solution of 
$$\Delta u+e^u=4\pi(\delta_p+\delta_{-p})\quad\text{on }E_{\tau}$$
if and only if $\tilde{u}(z):=u(z+\frac{\omega_k}{2})$ is a solution of 
$$\Delta u+e^u=4\pi(\delta_{p-\frac{\omega_k}{2}}+\delta_{-(p-\frac{\omega_k}{2})})\quad\text{on }E_{\tau}.$$
This implies $\Omega_{\tau}^{\mathbf 0}\cap \cup_{k=0}^3B(\frac{\omega_k}{2},\varepsilon)=\emptyset$.
The proof is complete.
\end{proof}

\subsection*{Acknowledgements}
Z. Chen was supported by National Key R\&D Program of China (Grant 2023YFA1010002) and NSFC (No. 1222109).


\begin{thebibliography}{99}                                                                                               %


\bibitem {Babich-Bordag}M.V. Babich and L.A. Bordag; \textit{Projective differential geometrical structure of the Painlev\'e equations}. J. Differ. Equ. \textbf{157} (1999), 452-485.

\bibitem {Bob-E}A. Bobenko and U. Eitner;\textit{ Bonnet surfaces and
Painlev\'{e} equations}. J. reine angew. Math. \textbf{499} (1998), 47-79.

\bibitem{B-M} H. Brezis and F. Merle; \textit{Uniform estimates and blow-up behavior for solutions of $-\Delta u =V(x)e^u$ in two dimensions}. Commun. Partial Differ. Equ. \textbf{16} (1991), 1223-1254.

\bibitem {CLMP}E. Caglioti, P. L. Lions, C. Marchioro and M. Pulvirenti;
\textit{A special class of stationary flows for two-dimensional Euler
equations: a statistical mechanics description.} Comm. Math. Phys.
\textbf{143} (1992), 501-525.



\bibitem {CLW} C.L. Chai, C.S. Lin and C.L. Wang; \textit{Mean field
equations, Hyperelliptic curves, and Modular forms: I}. Cambridge Journal of
Mathematics, \textbf{3} (2015), 127-274.



\bibitem {CL-1} C.C. Chen and C.S. Lin; \textit{Sharp estimates for solutions
of multi-bubbles in compact Riemann surfaces}. Comm. Pure Appl. Math.
\textbf{55} (2002), 728-771. 

\bibitem {CL-2} C.C. Chen and C.S. Lin; \textit{Topological degree for a mean
field equation on Riemann surfaces}. Comm. Pure Appl. Math. \textbf{56}
(2003), 1667-1727. 

\bibitem {CL-3} C.C. Chen and C.S. Lin; \textit{Mean field equation of
Liouville type with singular data: Topological degree}. Comm. Pure Appl. Math.
\textbf{68} (2015), 887-947. 


\bibitem {Chen-Kuo-Lin0}Z. Chen, T.J. Kuo and C.S. Lin; \textit{Hamiltonian
system for the elliptic form of Painlev\'{e} VI equation}. J. Math. Pures
Appl. \textbf{106} (2016), 546-581.



\bibitem {CKL1} Z. Chen, T.J. Kuo and C.S. Lin; \textit{The geometry of generalized Lam\'{e} equation, I}. J. Math. Pures Appl.  \textbf{127} (2019), 89-120.

\bibitem {CKL2} Z. Chen, T.J. Kuo and C.S. Lin; \textit{The geometry of generalized Lam\'{e} equation, II: Existence of pre-modular forms and application}. J. Math. Pures
Appl. \textbf{132} (2019), 251–272.

\bibitem{Chen-Kuo-Lin} Z. Chen, T.J. Kuo and C.S. Lin;
\textit{The geometry of generalized Lam\'{e} equation, III: one-to-one of the Riemann-Hilbert correspondence}. Pure Appl. Math. Q. \textbf{17} (2021), 1619–1668.

\bibitem {CL-AJM}Z. Chen and C.S. Lin; \textit{Sharp nonexistence results for curvature equations with four singular sources on rectangular tori}. Amer. J. Math. \textbf{142} (2020), 1269-1300.



\bibitem {Dubrovin-Mazzocco} B. Dubrovin and M. Mazzocco; \textit{Monodromy
of certain Painlev\'{e}-VI transcendents and reflection groups}. Invent. Math.
\textbf{141} (2000), 55-147.

\bibitem {EG-2015}A. Eremenko and A. Gabrielov; \textit{On metrics of curvature $1$ with four conic singularities on tori and on the sphere}.
Illinois J. Math. \textbf{59} (2015), 925-947.

\bibitem {EG}A. Eremenko and A. Gabrielov; \textit{Spherical Rectangles}.
Arnold Math. J. \textbf{2} (2016), 463-486.


\bibitem {Guzzetti} D. Guzzetti; \textit{The elliptic representation of the
general Painlev\'{e} VI equation}. Comm. Pure Appl. Math. \textbf{55} (2002),
1280-1363. 



\bibitem {Hit1} N. J. Hitchin; \textit{Twistor spaces, Einstein metrics and
isomonodromic deformations}. J. Differ. Geom. \textbf{42} (1995), no.1,
30-112.

\bibitem {GP} K. Iwasaki, H. Kimura, S. Shimomura and M. Yoshida;
\textit{From Gauss to Painlev\'{e}: A Modern Theory of Special Functions}.
Springer vol. E16, 1991.





\bibitem {LW} C.S. Lin and C.L. Wang; \textit{Elliptic functions, Green
functions and the mean field equations on tori}. Ann. Math. \textbf{172}
(2010), no.2, 911-954. 

\bibitem {LW2}C.S. Lin and C.L. Wang; \textit{Mean field
equations, Hyperelliptic curves, and Modular forms: II}.  J. \'{E}c. polytech. Math.
\textbf{4} (2017), 557-593.

\bibitem {LY}C.S. Lin and S. Yan; \textit{Existence of bubbling solutions for
Chern-Simons model on a torus.} Arch. Ration. Mech. Anal. \textbf{207} (2013), 353-392.

\bibitem {Lisovyy-Tykhyy} O. Lisovyy and Y. Tykhyy; \textit{Algebraic
solutions of the sixth Painlev\'{e} equation}. J. Geom. Phys. \textbf{85}
(2014), 124-163. 



\bibitem{MR} A. Malchiodi and D. Ruiz; \textit{New improved Moser-Trudinger inequalities and singular Liouville equations on compact surfaces}.
Geom. Funct. Anal. \textbf{21} (2011), 1196-1217.

\bibitem {NT1}M. Nolasco and G. Tarantello; \textit{Double vortex condensates
in the Chern-Simons-Higgs theory.} Calc. Var. PDE. \textbf{9} (1999), 31-94.



\bibitem {Okamoto1} K. Okamoto; \textit{Studies on the Painlev\'{e}
equations. I. Sixth Painlev\'{e} equation} $P_{VI}$. Ann. Mat. Pura Appl.
\textbf{146} (1986), 337-381. 









\bibitem {Takemura}K. Takemura; \textit{The Hermite-Krichever Ansatz for
Fuchsian equations with applications to the sixth Painlev\'{e} equation and to
finite gap potentials}. Math. Z. \textbf{263} (2009), 149-194.




\end{thebibliography}
\end{document}